\documentclass[11pt]{article}
\usepackage{amssymb,amsmath,amsthm,bm,citesort,setspace}

\newcommand{\R}{\mathbf{R}}

\renewcommand{\P}{\mathrm{P}}
\newcommand{\E}{\mathrm{E}}

\newcommand{\1}{\boldsymbol{1}}
\renewcommand{\d}{{\rm d}}
\renewcommand{\Re}{\mathrm{Re}}
\newcommand{\e}{{\rm e}}

\author{Daniel Conus\\University of Utah
\and Davar Khoshnevisan\\University of Utah}
\title{\bf Weak nonmild solutions to some SPDEs\thanks{%
	Research supported in part by the
	Swiss National Science Foundation Fellowship PBELP2-122879 (D.C.)
	and the NSF grant DMS-0706728 (D.K.).}}
\date{April 15, 2010}
\newtheorem{stat}{Statement}[section]
\newtheorem{proposition}[stat]{Proposition}

\newtheorem{theorem}[stat]{Theorem}
\newtheorem{lemma}[stat]{Lemma}
\theoremstyle{definition} 

\newtheorem{remark}[stat]{Remark}

\newtheorem{example}[stat]{Example}

\numberwithin{equation}{section}
\begin{document}\onehalfspacing
\maketitle
\begin{abstract}
We study the nonlinear stochastic heat equation driven by space-time 
white noise in the case that the initial datum $u_0$ is a 
(possibly signed) measure. In this case, one cannot obtain a 
mild random-field solution in the usual sense. We prove instead
that it is possible to establish the existence and uniqueness 
of a weak solution with values in a suitable function space. 
Our approach is based on a construction of 
a generalized definition of a stochastic convolution 
via Young-type inequalities. 
\end{abstract}

\section{Introduction}

Let us consider the nonlinear stochastic heat equation
\begin{equation}\label{heat}
	\frac{\partial}{\partial t}u_t(x)=(\mathcal{L}u_t)(x)+
	\sigma(u_t(x))\dot{W}(t\,,x) \qquad (t \geq 0, x \in \R),
\end{equation}
where: 
(i) $\mathcal{L}$ is the generator of a real-valued L\'evy process 
$\{X_t\}_{t \geq 0}$ with L\'evy exponent $\Psi$, normalized so that
$\E\e^{i\xi X_t}=\e^{-t\Psi(\xi)}$ for every $\xi\in\R$ and $t\ge 0$;
(ii) $\sigma:\R\to\R$ is Lipschitz continuous with Lipschitz constant 
${\rm Lip}_\sigma$; (iii) $\dot{W}$ is space-time white noise; and
(iv) The initial datum $u_0$ is a signed Borel measure on $\R$. 

Equation \eqref{heat} arises in many different contexts;
three notable examples are Bertini and Cancrini \cite{BC}, Gy\"ongy and 
Nualart \cite{GN}, and Carmona and Molchanov \cite{CM94}.

In the case that $u_0:\R\to\R_+$ is a bounded measurable function, 
the theory of Dalang \cite{Dalang} shows that there exists a unique 
random-field mild solution $\{u_t(x)\}_{t \geq 0, x \in \R}$ provided that
\begin{equation}\label{Dalang}
	\Upsilon(\beta):= \frac{1}{2\pi} \int_{-\infty}^\infty 
	\frac{\d\xi}{\beta+2\Re\Psi(\xi)}<\infty \quad
	\text{for some, hence all, $\beta>0$}.
\end{equation}
In general, Dalang's Condition  \eqref{Dalang} cannot be improved upon 
\cite{Dalang,PZ}. 

Dalang's condition \eqref{Dalang} implies also
that the L\'evy process $X$ has transition functions
$p_t(x)$ \cite[Lemma 8.1]{FKN}; i.e., for all
measurable $f:\R\to\R_+$,
\begin{equation}
	\E f(X_t)= \int_{-\infty}^\infty p_t(z)f(z)\,\d z
	\qquad\text{for all $t>0$}.
\end{equation}
A \emph{mild solution} in this setting is a random field 
$\{u_t(x)\}_{t \geq 0, x \in \R}$ that satisfies
\begin{equation} \label{mild_sol}
	u_t(x) = (P_tu_0)(x) + \int_{[0,t]\times\R} p_{t-s}(y-x) \sigma(u_s(y))\,W(\d s\,\d y)
	\quad\text{a.s.},
\end{equation}
for all $t \geq 0$ and $x \in \R$, where $\{P_t\}_{t \geq 0}$ 
denotes the semigroup associated to the process $X$, 
and the stochastic integral is understood as a
Walsh martingale-measure stochastic integral \cite{walsh}. 
Notice that \eqref{mild_sol} can be rewritten in the following form:
For all $(t\,,x)\in\R_+\times\R$,
\begin{equation} \label{mild_star_form}
	u_t(x) = (P_tu_0)(x) + \left( \tilde{p}*(\sigma\circ u)\dot W \right)_t(x)
	\quad\text{a.s.},
\end{equation}
where ``$*$'' denotes ``stochastic convolution''; see \eqref{eq:SC} below and $\tilde{p}_t(x) = p_t(-x)$ for all $x \in \R$.

In the case that $u_0$ is not a bounded and measurable function, but 
instead a (possibly signed) Borel measure on $\R$, the solution $u$ 
cannot be defined as a random field, but has to be considered as a 
process of  function-space type. As a consequence, the stochastic 
convolution in \eqref{mild_sol} is not well defined in the sense of 
Walsh. Section \ref{sec:gen_conv} below is devoted to extending 
the definition of the stochastic convolution of a process $\Gamma$ 
with respect to $Z\dot{W}$ in the case that $Z$ is in a suitable 
Banach space $\bm{B}^k_{\beta,\eta}$ of random processes. 
The key step of this extension involves developing a kind of
``stochastic Young inequality''
(Proposition \ref{pr:SY}). Such an inequality appeared earlier
in \cite{Conus_Khoshnevisan}, in a different context, in order to obtain intermittency 
properties for equation \eqref{heat} in the case that $u_0$ is a lower semicontinuous
bounded function of compact support.

In Section \ref{sec:eu} we establish the existence and uniqueness 
of a weak solution to \eqref{heat}. Namely, we prove that
Dalang's condition \eqref{Dalang} implies that if 
$u_0 = \mu$ is a (possibly signed) Borel measure on $\R$ 
that satisfies a suitable integrability condition \eqref{cond:u0},
then there exists a unique $u \in \bm{B}^k_{\beta,\eta}$ such that
$u$ almost surely satisfies \eqref{mild_star_form} for almost every 
$t \geq 0$ and $x \in \R$. This solution is \emph{not} a random field;
but it \emph{is} a function-space-valued solution. 

In Section \ref{sec:sharp} we prove that our condition for existence 
and uniqueness is unimprovable. And in Section \ref{sec:last} we 
mention briefly examples of initial data $u_0$ that lead to the
existence and uniqueness of a weak solution to \eqref{heat}, together
with further remarks that explain what happens if we study
the 1-D stochastic wave equation in place of \eqref{heat}.

\section{Generalized stochastic convolutions} \label{sec:gen_conv}

Let $\Gamma:(0\,,\infty)\times\R\to\R$ be measurable, and 
$Z:=\{Z_t(x)\}_{t>0,x\in\R}$ be a predictable random field
in the sense of Walsh \cite[p.\ 292]{walsh}. Let us define
the \emph{stochastic convolution} 
$\Gamma*Z\dot{W}$ of the process $\Gamma$ with 
the noise $Z\dot{W}$ as the predictable random field
\begin{equation}\label{eq:SC}
	(\Gamma*Z\dot{W})_t(x) := \int_{[0,t]\times\R} \Gamma_{t-s}(x-y)
	Z_s(y)\,W(\d s\,\d y).
\end{equation}
The preceding is defined as a stochastic integral with respect to 
the martingale measure $Z\dot{W}$
in the sense of Walsh \cite[Theorem 2.5]{walsh}, and is well defined
in the sense of Walsh \cite[Chapter 2]{walsh} provided that
the following condition holds for all $t>0$ and $x\in\R$:
\begin{equation}\label{eq:Ito:Isometry}
	\left\| (\Gamma*Z\dot{W})_t(x)\right\|_2^2 = \int_0^t\d s\int_{-\infty}^\infty
	\d y\ \left[ \Gamma_{t-s}(x-y)\right]^2 \|Z_s(y)\|_2^2<\infty.
\end{equation}

Let $\bm{W}^2$ denote the collection of all predictable random fields
$Z$ that satisfy the following: $Z_t(x)\in L^2(\P)$ and
\begin{equation}\label{eq:Ito:Isometry2}
	\int_0^t\d s\int_{-\infty}^\infty
	\d y\ \left[ \Gamma_{\tau-s}(x-y)\right]^2 \|Z_s(y)\|_2^2<\infty,
\end{equation}
for all $0<t\le\tau$ and $x\in\R$.

We may think of the elements of $\bm{W}^2$
as \emph{Walsh-integrable random fields}. And because
\eqref{eq:Ito:Isometry2} implies \eqref{eq:Ito:Isometry},
the preceding discussion
tells us that the stochastic convolution $\Gamma*Z\dot{W}$ is
a well-defined predictable random field for every $Z\in\bm{W}^2$.

Our present goal is to extend the definition of the stochastic convolution
of $Z$ so that it includes more general random processes $Z$. 
Other extensions of this stochastic convolutions have 
been developed for other purposes as well
\cite{Dalang,DalangMueller,NualartQuer,ConusDalang}.

Let us choose and fix a real number $k\in[2\,,\infty)$, and define
$\bm{L}^k$ to be the collection of all predictable
random fields $\{Z_t(x)\}_{t> 0,x\in\R}$ such that
$Z_t(x)\in L^k(\P)$ for all $t> 0$ and $x\in\R$. 
Let $M(\R)$ be the space of $\sigma$-finite Borel measures on $\R$. For every $\beta>0$
and $\eta\in M(\R)$ we can define
a norm on $\bm{L}^k$ as follows: For every $Z\in\bm{L}^k$,
\begin{equation} \label{norm}
	\mathcal{N}^k_{\beta,\eta}(Z) := 
	\left(\int_{0}^{\infty} \e^{-\beta t}\,\d t\  
	\sup_{z \in \R } \int_{-\infty}^\infty \eta(\d x )\ 
	\left\| Z_t(x-z)\right\|_k^2\right)^{1/2}.
\end{equation}

Here and throughout, we use implicitly the following
observation: If $Z,Z'\in\bm{L}^k$ satisfy
$\mathcal{N}^k_{\beta,\eta}(Z-Z')=0$, then
$Z$ and $Z'$ are modifications of one another. There is
an obvious converse as well: If $Z$ and $Z'$ are modifications of one another, 
then $\mathcal{N}^k_{\beta,\eta}(Z-Z')=0$. We omit the elementary proof.

Our next proposition is a ``stochastic Young's inequality,''
and plays a key role in our extension of Walsh-type stochastic convolutions.
But first let us introduce some notation that will be used
here and throughout.

Throughout this paper, $z_k$ denotes the optimal constant
in Burkholder's $L^k(\P)$-inequality for continuous square-integrable martingales; 
its precise value involves zeroes of Hermite polynomials,
and has been computed by Davis \cite{Davis}. 
By the It\^o isometry, $z_2=1$. Carlen and Kree \cite[Appendix]{CK} have shown 
that $z_k \le 2\sqrt{k}$ for all $k \ge 2$, and moreover $z_k=(2+o(1))\sqrt{k}$
as $k\to\infty$.

\begin{proposition}[A stochastic Young's inequality]\label{pr:SY}
	For every $k\in[2\,,\infty)$,
	$Z\in \bm{W}^2\cap \bm{L}^k$, $\eta\in M(\R)$, and $\beta>0$, 
	\begin{equation}
		\mathcal{N}^k_{\beta,\eta}(\Gamma*Z\dot W)
		\le z_k\left(\int_0^\infty\e^{-\beta t} \|\Gamma_t\|_{L^2(\R)}^2\,\d t
		\right)^{1/2}\cdot
		\mathcal{N}^k_{\beta,\eta}(Z),
	\end{equation}
where $z_k$ is the optimal constant in Burkholder's inequality.
\end{proposition}

\begin{remark}
	We emphasize that $\bm{W}^2\cap \bm{L}^2=\bm{W}^2$.
\qed\end{remark}
Before we prove Proposition \ref{pr:SY} let us first describe how it can be used
to extend stochastic convolutions. Proposition \ref{pr:SY} will be proved
after that extension is described.

Let $\bm{B}^k_{\beta,\eta}$ denote the completion of $\bm{W}^2\cap
\bm{L}^k$ under the norm $\mathcal{N}^k_{\beta,\eta}$.\footnote{The latter
is of course a norm on equivalence classes of modifications of random fields
and not on random fields themselves. But we abuse notation as it is standard.} 
It follows then that
$\bm{B}^k_{\beta,\eta}$ is a Banach space of predictable
processes [identified up to evanescence].

Proposition \ref{pr:SY} immediately implies that if
\begin{equation}\label{cond:Dalang}
	\Upsilon(\beta) :=
	\int_0^\infty \e^{-\beta t}\|\Gamma_t\|_{L^2(\R)}^2\,\d t<\infty,
\end{equation}
then $Z\mapsto \Gamma*Z\dot{W}$ has a unique extension to 
all $Z\in\bm{B}^k_{\beta,\eta}$, and the resulting extension---written
still as $Z\mapsto \Gamma*Z\dot{W}$---defines a bounded linear operator from
$\bm{B}^k_{\beta,\eta}$ into itself. And the operator norm is at most
the square root of the \emph{Dalang integral} $\Upsilon(\beta)$.
[In the case that $\Gamma_t(x)$ denotes the transition density of a
L\'evy process with L\'evy exponent $\Psi$, Plancherel's theorem
implies that $\Upsilon(\beta)$ is the same Dalang integral as
in \eqref{Dalang}; see \eqref{eq:L2:Upsilon} below as well.]

From now on, we deal solely with this extension of the stochastic convolution.
However, we point out the there is a great deal of variability in this extension,
as the parameters $\beta>0$, $k\in[2\,,\infty)$, and $\eta\in M(\R)$
can take on many different values.

Let us conclude this section by establishing our stochastic Young's inequality.

The proof of Proposition \ref{pr:SY} relies on an elementary estimate for Walsh-type
stochastic integrals.

\begin{lemma}\label{lem:stoch:Minkowski}
	For all real numbers $t>0$, $x\in\R$, and
	$k\in[2\,,\infty)$, and for every $Z\in\bm{W}^2\cap \bm{L}^k$,
	\begin{equation}
		\left\| (\Gamma*Z\dot W)_t(x)\right\|_k^k
		\le z_k^k\left(\int_0^t\d s\int_{-\infty}^\infty\d y\
		[\Gamma_{t-s}(x-y)]^2\| Z_s(y)\|_k^2\right)^{k/2}.
	\end{equation}
\end{lemma}

\begin{proof}
	Condition \eqref{eq:Ito:Isometry2} implies that if $0<t\le\tau$, then
	\begin{equation}
		(\Gamma*Z\dot{W})_{t,\tau}(x):=\int_{[0,t]\times\R}
		\Gamma_{\tau-s}(x-y)Z_s(y)\, W(\d s\,\d y)
	\end{equation}
	is well defined and in $L^2(\P)$. Moreover,
	\begin{equation}
		\left\| (\Gamma*Z\dot{W})_{t,\tau}(x)\right\|_2=
		\left(\int_0^t\d s\int_{-\infty}^\infty \d y\
		\left[ \Gamma_{\tau-s}(x-y)\right]^2 \| Z_s(y)\|_2^2\right)^{1/2}.
	\end{equation}
	Walsh's theory of martingale measures \cite[Theorem 2.5]{walsh} tells us that
	the stochastic process $(0\,,\tau]\ni t\mapsto (\Gamma*Z\dot{W})_{t,\tau}(x)$
	is a continuous $L^2(\P)$-martingale.
	Therefore,  Davis's refinement \cite{Davis} of Burkholder's inequality 
	\cite{Burkholder,BDG,Millar} implies that
	\begin{equation}
		\left\| (\Gamma*Z\dot W)_{t,\tau}(x)\right\|_k^k
		\le z_k^k \left\|\int_0^t\d s\int_{-\infty}^\infty\d y\,
		\left[\Gamma_{\tau-s}(x-y)\right]^2
		\left[ Z_s(y)\right]^2 \right\|_{k/2}^{k/2}.
	\end{equation}
	And it follows from Minkowski's inequality that
	\begin{equation}\label{eq:BI}
		\left\| (\Gamma*Z\dot W)_{t,\tau}(x)\right\|_k^k
		\le z_k^k \int_0^t\d s\int_{-\infty}^\infty\d y\,
		\left[\Gamma_{\tau-s}(x-y)\right]^2
		\left\| Z_s(y) \right\|_k^2.
	\end{equation}
	The lemma follows from this upon setting $\tau:=t$.
\end{proof}

\begin{proof}[Proof of Proposition \ref{pr:SY}]
	The original construction of Walsh implies that
	$\|(\Gamma*Z\dot{W})_t(x)\|_k$ defines a Borel-measurable function of
	$(t\,,x)\in(0\,,\infty)\times\R$. Indeed, it suffices to verify this
	measurability assertion in the case that $Z$ is a simple function in the sense
	of Walsh \cite[p.\ 292]{walsh}, in which case the said measurability follows from a
	direct computation. 
	
	We may apply Lemma \ref{lem:stoch:Minkowski}
	with $x-z$ in place of the variable $x$, and then integrate $[\d\eta]$ to obtain
	\begin{equation}\begin{split}
		&\int_{-\infty}^\infty\eta(\d x)\ \left\| (\Gamma*Z\dot W)_t(x-z)\right\|_k^2\\
		&\le z_k^2\int_{-\infty}^\infty\eta(\d x)\int_0^t\d s
			\int_{-\infty}^\infty\d y\ [\Gamma_{t-s}(x-z-y)]^2\|Z_s(y)\|_k^2\\
		&= z_k^2\int_{-\infty}^\infty\eta(\d x)\int_0^t\d s
			\int_{-\infty}^\infty\d y\ [\Gamma_{t-s}(y)]^2\|Z_s(x-z-y)\|_k^2\\
		&\le z_k^2\int_0^t\d s\ \|\Gamma_{t-s}\|_{L^2(\R)}^2
			\sup_{v\in\R}\int_{-\infty}^\infty\eta(\d x)\ \|Z_s(x-v)\|_k^2.
	\end{split}\end{equation}
	Or equivalently,
	\begin{equation}\begin{split}
		&\sup_{z\in\R}\int_{-\infty}^\infty\eta(\d x)\ 
			\left\| (\Gamma*Z\dot W)_t(x-z)\right\|_k^2\\
		&\hskip.9in\le z_k^2\int_0^t\d s\ \|\Gamma_{t-s}\|_{L^2(\R)}^2
			\sup_{z\in\R}\int_{-\infty}^\infty\eta(\d x)\ \|Z_s(x-z)\|_k^2.
	\end{split}\end{equation}
	Multiply both sides by $\exp(-\beta t)$, integrate $[\d t]$ and use Laplace transforms properties for convolutions to obtain the result.
\end{proof}

\begin{proposition}\label{pr:SI:Lip}
	Suppose $\sigma:\R\to\R$ is Lipschitz continuous and $Z,Z^*\in\bm{B}^k_{\beta,\eta}$
	for some $k\in[2\,,\infty)$, $\beta>0$, and 
	$\eta \in M(\R)$.
	Then,
	\begin{equation}
		\mathcal{N}^k_{\beta,\eta}(\sigma\circ Z - \sigma\circ Z^*) \le
		\text{\rm Lip}_\sigma\cdot\mathcal{N}^k_{\beta,\eta}(Z-Z^*).
	\end{equation}
\end{proposition}

\begin{proof}
	If $Z,Z^*\in\bm{W}^2\cap \bm{L}^k$, then this is immediate.
	In the general case we proceed by approximation:
	Let $Z^1,Z^2,\dots,Z^{1,*},Z^{2,*},\ldots$ be in $\bm{W}^2\cap\bm{L}^k$
	such that $Z^n\to Z$ and $Z^{n,*}\to Z^*$ in $\bm{B}^k_{\beta,\eta}$,
	as $n\to\infty$. By going to a subsequence, if necessary, we can
	[and will!] also assume that
	\begin{equation}
		\mathcal{N}^k_{\beta,\eta}(Z^n-Z^{n+1}) +
		\mathcal{N}^k_{\beta,\eta}(Z^{n,*}-Z^{n+1,*})\le 2^{-n}
		\quad\text{for all $n\ge 1$}.
	\end{equation}
	It follows also that for all $n\ge 1$,
	\begin{equation}\label{eq:est}
		\mathcal{N}^k_{\beta,\eta}(\sigma\circ Z^n-\sigma\circ Z^{n+1}) +
		\mathcal{N}^k_{\beta,\eta}(\sigma\circ Z^{n,*}-\sigma\circ Z^{n+1,*})\le 
		\text{Lip}_\sigma\cdot 2^{-n}.
	\end{equation}
	Of course, this implies immediately that $\sigma\circ Z^n$
	and $\sigma\circ Z^{n,*}$ converge in $\bm{B}^k_{\beta,\eta}$.
	It suffices to prove that the mentioned limits are respectively
	$\sigma\circ Z$ and $\sigma\circ Z^*$. But \eqref{eq:est} implies
	that 
	\begin{equation}\label{eq:prec}
		\int_0^\infty\e^{-\beta t}\,\d t
		\sum_{n=1}^\infty \sup_{z\in\R}\int_{-\infty}^\infty
		\eta(\d x)\left\| \Delta^n_t(x-z)\right\|_k^2
		<\infty,
	\end{equation}
	where $\Delta^n_t(x)$ stands for any one of the following four quantities:
	\begin{itemize}
		\item $Z^n_t(x)-Z^{n+1}_t(x)$; 
		\item $Z^{n,*}_t(x)-Z^{n+1,*}_t(x)$;
		\item $\sigma(Z^n_t(x))-\sigma(Z^{n+1}_t(x))$; or
		\item $\sigma(Z^{n,*}_t(x))-\sigma(Z^{n+1,*}_t(x))$. 
	\end{itemize}
	Because 
	$\sum_{n=1}\sup_{z\in\R}(\,\cdots)\le\sup_{z\in\R}\sum_{n=1}^\infty(\,\cdots)$
	in \eqref{eq:prec}, it follows readily that for almost every pair $(t\,,x)\in\R_+\times\R$:
	\begin{itemize}
		\item $\lim_{n\to\infty}Z^n_t(x)=Z_t(x)$ almost surely;
		\item $\lim_{n\to\infty}Z^{n,*}_t(x)=Z^*_t(x)$ almost surely;
		\item $\lim_{n\to\infty}\sigma(Z^n_t(x))=\sigma(Z_t(x))$ almost surely; and
		\item $\lim_{n\to\infty}\sigma(Z^{n,*}_t(x))=\sigma(Z^*_t(x))$ almost surely.
	\end{itemize}
	[Note the order of the quantifiers!] We showed earlier that
	$\lim_{n\to\infty}\sigma\circ Z^n$ and $\lim_{n\to\infty}\sigma\circ Z^{n,*}$
	exist in $\bm{B}^k_{\beta,\eta}$. The preceding shows that those
	limits are respectively $\sigma\circ Z$ and $\sigma\circ Z^*$.
	This completes the proof.
\end{proof}

\section{Existence and Uniqueness} \label{sec:eu}

This section is devoted to the statement and proof of the existence and 
uniqueness of a weak solution to \eqref{heat}. We will make 
use of the generalized stochastic convolution developed in 
Section \ref{sec:gen_conv}. 

Before we proceed further, let us observe that from now on
$\Gamma_t(x)$ of the previous section is
chosen to be equal to the modified transition functions
$\tilde{p}_t(x)$, in which case Dalang's integral can be computed
from Plancherel's formula as follows:
\begin{equation}\label{eq:L2:Upsilon}
	\Upsilon(\beta)=\int_0^\infty \e^{-\beta t}\|p_t\|_{L^2(\R)}^2\,\d t
	=\frac{1}{2\pi}\int_{-\infty}^\infty \frac{\d\xi}{\beta+2\Re\Psi(\xi)}.
\end{equation}
In particular, the $\Upsilon$ of
\eqref{cond:Dalang} and that of \eqref{Dalang} are equal
in the present setting.

Next we identify our notion of ``solution'' to \eqref{heat}
in the case that $u_0=\mu$ is a measure.

Suppose first that $u_0$ is a nice \emph{function}
and \eqref{heat} has a mild solution $u$ with initial datum $u_0$.
Then for all $t>0$ and $x\in\R$,
\begin{equation}
	\P\left\{
	u_t(x) = (P_tu_0)(x) +\left( \tilde{p}*(\sigma\circ u)\dot W \right)_t(x)
	\right\}=1.
\end{equation}
Consequently, Fubini's theorem tells us that every mild 
solution $u$ to \eqref{heat} with initial function $\mu:=u_0$ is a 
\emph{weak solution} in the sense that the following holds with probability one
[note the order of the quantifiers!]:
\begin{equation} \label{mild}
	u_t(x) = (P_t\mu)(x) +\left( \tilde{p}*(\sigma\circ u)\dot W \right)_t(x)
	\quad\text{for a.e.\ $(t\,,x)\in\R_+\times\R.$}
\end{equation}
It is easy to see that the preceding agrees with Walsh's definition 
of a weak solution \cite[p.\ 309]{walsh}.

Now we consider \eqref{heat} in the case that $u_0=\mu$ is a 
possibly-signed Borel measure on $\R$.

Let us suppose that Dalang's condition \eqref{cond:Dalang} holds, and consider
an arbitrary $u\in\bm{B}^k_{\beta,\eta}$. Since $\sigma$ is
Lipschitz continuous, it follows that
$\sigma\circ u\in \bm{B}^k_{\beta,\eta}$. Therefore,
we can conclude that the stochastic convolution 
$\tilde{p}*(\sigma\circ u)\dot{W}$ is a well-defined mathematical object,
as was shown in the previous section. Consequently, 
we can try to find a solution $u$ to \eqref{heat} with
$u_0=\mu$ by 
seeking to find $u\in\bm{B}^k_{\beta,\eta}$ such that
\begin{equation}\label{mild2}
	u = P_\bullet \mu + \tilde{p}*(\sigma\circ u)\dot W,
\end{equation}
where the equality is understood as equality of elements of
$\bm{B}^k_{\beta,\eta}$. Of course, we implicitly are
assuming that $P_\bullet \mu\in\bm{B}^k_{\beta,\eta}$ as well.
That condition is clearly satisfied if
\begin{equation}\label{cond:u0}
	\int_0^\infty \e^{-\beta s}\,\d s\,
	\sup_{z\in\R}\int_{-\infty}^\infty\eta(\d x)\left|
	(P_s\mu)(x-z)\right|^2<\infty.
\end{equation}
Then, $u$ is a solution of function-space type
to \eqref{heat} with $u_0=\mu$. But it has more structure than that.
Indeed, suppose that:
(i) \eqref{cond:u0} holds; and (ii)
There exists $u\in\bm{B}^k_{\beta,\eta}$ that satisfies
\eqref{mild2}. Then the preceding discussion shows
also that $u$ is a weak solution to \eqref{heat}
in the sense of Walsh \cite[p.\ 309]{walsh}. And
it would be hopeless to try to prove that such
a $u$ is a mild solution, as there is no natural way to
define $u_t(x)$ for all $t>0$ and $x\in\R$.

Throughout the remainder of this section
we choose $\eta \in M(\R)$. In the case that $\sigma(0)\neq 0$, then we
assume additionally that $\eta$ is a finite measure.

\begin{theorem}\label{th:1}
	Consider \eqref{heat}
	subject to $u_0=\mu$, where $\mu$ is a signed
	measure that satisfies \eqref{cond:u0}. 
	If \eqref{Dalang} holds
	and
	\begin{equation}\label{cond:beta:large}
		\Upsilon(\beta)<\frac{1}{(z_k\text{\rm Lip}_\sigma)^2},
	\end{equation}
	then there exists
	a solution $u\in \bm{B}^k_{\beta,\eta}$
	that satisfies \eqref{mild2}. Moreover, $u$ is unique
	in $\bm{B}^k_{\beta,\eta}$; i.e., if there exists another
	weak solution $v$ that is in $\bm{B}^k_{\beta,\eta}$ for
	some $k\ge 2$, then $v$ is a modification of $u$.
\end{theorem}

\begin{proof}
	First, we argue that we can always choose $\beta$ such that
	\eqref{cond:beta:large} holds.
	
	Indeed, Condition \eqref{cond:u0} implies
	that $\mathcal{N}^k_{\beta,\eta}(P_\bullet u_0)<\infty$
	for all $\beta> 0$ and $k\in[2\,,\infty)$. 
	Also, because of Dalang's condition \eqref{Dalang}, and
	by the monotone convergence theorem,
	$\lim_{\alpha\to\infty}\Upsilon(\alpha)=0$. Therefore, we can combine these
	two observations to deduce that \eqref{cond:beta:large} holds for all $\beta$ large,
	where $1/0:=\infty$. Throughout the remainder of the proof, we hold fixed a
	$\beta$ that satisfies \eqref{cond:beta:large}.
	
	Set $u^{(0)}_t := 0$, and iteratively define
	\begin{equation}
		u^{(n+1)} := P_\bullet\mu + \tilde{p}*\left(
		[\sigma\circ u^{(n)}]\dot W\right).
	\end{equation}
	These $u^{(n+1)}$'s are all well defined elements of 
	$\bm{B}^k_{\beta,\eta}$. In fact,
	it follows from Proposition \ref{pr:SY} that for all $n\ge 0$,
	\begin{equation}\begin{split}
		\mathcal{N}^k_{\beta,\eta}\left(u^{(n+1)}\right)
			&\le \mathcal{N}^k_{\beta,\eta}\left(P_\bullet\mu\right)+
			z_k \sqrt{\Upsilon(\beta)}\,\mathcal{N}^k_{\beta,\eta}
			\left(\sigma\circ u^{(n)}\right).
	\end{split}\end{equation}
	And because $|\sigma(z)|\le |\sigma(0)| + \text{Lip}_\sigma|z|$ for all
	$z\in\R$,
	\begin{equation}\begin{split}
		&\mathcal{N}^k_{\beta,\eta}\left(u^{(n+1)}\right)\\
		&\hskip.1in\le \mathcal{N}^k_{\beta,\eta}\left(P_\bullet\mu\right)+
			z_k \sqrt{\Upsilon(\beta)}\left[|\sigma(0)|\cdot \mathcal{N}^k_{\beta,\eta}
			(\1) +\text{Lip}_\sigma\cdot
			\mathcal{N}^k_{\beta,\eta}\left(u^{(n)}\right)\right],
	\end{split}\end{equation}
	where $\1_t(x):=1$ for all $t>0$ and $x\in\R$. 
	In particular, $u^{(l)}\in \bm{B}^k_{\beta,\eta}$ for all $l\ge 0$, by induction.
	This is clear if $\sigma(0)=0$; and if $\sigma(0)\neq 0$, then it is also
	true because $\mathcal{N}^k_{\beta,\eta}(\1) =\sqrt{\eta(\R)/\beta}<\infty$,
	thanks to the finiteness assumption on $\eta$ [for the case $\sigma(0)\neq 0$].
	Moreover, \eqref{cond:beta:large} and induction together show more; namely, that
	$\sup_{n\ge 0}\mathcal{N}^k_{\beta,\eta}(u^{(n)})<\infty$.
	
	A similar computation, this time using also Proposition \ref{pr:SI:Lip}, shows that for all $n\ge 1$,
	\begin{equation}
		\mathcal{N}^k_{\beta,\eta}\left(u^{(n+1)}-u^{(n)}\right)
		\le z_k\text{Lip}_\sigma\sqrt{\Upsilon(\beta)}\cdot
		\mathcal{N}^k_{\beta,\eta}\left(u^{(n)}-u^{(n-1)}\right).
	\end{equation}
	And \eqref{cond:beta:large} implies that 
	$\sum_{n=0}^\infty \mathcal{N}^k_{\beta,\eta} (u^{(n+1)}-u^{(n)} )
	<\infty$,
	therefore, $\{u^{(n)}\}_{n=0}^\infty$ is a Cauchy sequence in $\bm{B}^k_{\beta,\eta}$.
	Let $u:=\lim_{n\to\infty}u^{(n)}$, where the limit takes place in $\bm{B}^k_{\beta,\eta}$.
	According to Proposition \ref{pr:SY}, 
	\begin{equation}\begin{split}
		\mathcal{N}^k_{\beta,\eta}\left(
			\tilde{p}*u^{(n)}\dot W - \tilde{p}* u\dot W\right) &\le
			z_k\sqrt{\Upsilon(\beta)}\cdot \mathcal{N}^k_{\beta,\eta}\left(
			u^{(n)}-u\right)\\
		&\to 0\qquad\text{as $n\to\infty$}.
	\end{split}\end{equation}
	It follows readily from these remarks that
	$\mathcal{N}^k_{\beta,\eta} (
	u - P_\bullet \mu+ \tilde{p}*(\sigma\circ u)\dot{W} )=0$.
	That is, $u$ satisfies \eqref{mild} for almost all $(t\,,x)\in\R_+\times\R$;
	see also \eqref{mild2}. This proves part (1) of the theorem.
	
	In order to prove the second part, let us suppose that there exists
	another ``weak solution'' $v\in\bm{B}^k_{\beta,\eta}$. Then,
	$\delta:=u-v\in \bm{B}^k_{\beta,\eta}$ and
	\begin{equation}
		\delta = \tilde{p}*\left([\sigma\circ u]\dot{W}\right)-
		\tilde{p}*\left([\sigma\circ v]\dot{W}\right)=
		\tilde{p}*\left([\sigma\circ u-\sigma \circ v]\dot{W}\right).
	\end{equation}
	[The second identity follows from the very construction of our
	stochastic convolution, using the fact that $Z\mapsto \tilde{p}* Z\dot{W}$
	is a bounded \emph{linear} map from $\bm{B}^k_{\beta,\eta}$
	to itself.]
	Propositions \ref{pr:SY} and \ref{pr:SI:Lip}
	together imply the following:
	\begin{equation}\begin{split}
		\mathcal{N}^k_{\beta,\eta}(\delta) &\le z_k\sqrt{\Upsilon(\beta)}\cdot
			\mathcal{N}^k_{\beta,\eta}\left(\sigma\circ u - \sigma\circ v\right)\\
		&\le z_k\text{Lip}_\sigma\sqrt{\Upsilon(\beta)}\cdot \mathcal{N}^k_{\beta,\eta}
			(\delta).
	\end{split}\end{equation}
	Thanks to \eqref{cond:beta:large}, $\mathcal{N}^k_{\beta,\eta}(u-v)=
	\mathcal{N}^k_{\beta,\eta}(\delta)=0$. This readily implies that
	$u$ and $v$ are modifications of one another, as well.
\end{proof}

\section{On Condition \eqref{cond:beta:large}} \label{sec:sharp}

Let us consider the measure $\eta_m \in M(\R)$ defined by
\begin{equation}
	\eta_m(\d x) = \e^{-|x|/m}\,\d x,
\end{equation}
where $m>0$ is fixed. If $\sigma(0)=0$, then we may take $m:=\infty$,
whence $\eta(\d x)=\d x$.

\begin{theorem}\label{th:2}
	Suppose \eqref{heat} has a solution $u\in\cap_{m>0}\bm{B}^2_{\beta,\eta_m}$
	with $u_0=\mu$ for a nonvoid signed Borel measure $\mu$ on $\R$ with
	$|\mu|(\R)<\infty$. Suppose
	also that ${\rm L}_\sigma:=\inf_{z\in\R}|\sigma(z)/z|>0$. Then, $\beta$ satisfies
	$\Upsilon(\beta) < {\rm L}_\sigma^{-2}$.
\end{theorem}

\begin{proof} Let $\mathcal{M}_{\beta}$ be the norm defined by
	\begin{equation}
		\mathcal{M}_\beta(Z) := 
		\left(\int_{0}^{\infty} \e^{-\beta t}\,\d t
		\int_{-\infty}^\infty \e^{-|x|/m}\,\d x\
		\left\| Z_t(x)\right\|_2^2\right)^{1/2}.
	\end{equation}
	Notice that $M_{\beta}$ is similar to $\mathcal{N}_{\beta,\eta_m}$ , but is
	missing a supremum on $Z$ in the space variable; cf.\ \eqref{norm}.
	Moreover, $\mathcal{M}_\beta(u)\le\mathcal{N}^2_{\beta,\eta_m}(u)<\infty$.
	Note that if $H,Z\in\bm{B}^2_{\beta,\eta_m}$ with one of 
	them---say $H$---random and the other one deterministic, then we have
	$[\mathcal{M}_\beta(H+G)]^2=
	[\mathcal{M}_\beta(H)]^2 +  [\mathcal{M}_\beta(G) ]^2$.
	This is a direct computation if $H,G\in\bm{W}^2$; the general case
	follows from approximation (we omit the details because the 
	method appears already
	during the course of the proof of Proposition \ref{pr:SI:Lip}).
	It follows that
	\begin{equation} \label{Mu}
		\left[ \mathcal{M}_\beta(u) \right]^2
		= \left[\mathcal{M}_\beta(P_\bullet\mu)\right]^2
		+\left[\mathcal{M}_\beta\left( \tilde{p}*(
		[\sigma\circ u]\dot{W})\right)\right]^2.
	\end{equation}
	The method of proof of Proposition \ref{pr:SI:Lip}, together with
	the simple bound, $\e^{-|x|/m}\ge\e^{-|x-y|/m}\cdot\e^{-|y|/m}$,
	shows also that
	\begin{equation}
		\mathcal{M}_\beta\left( \tilde{p}*([\sigma\circ u]\dot{W})\right) \ge
		{\rm L}_\sigma\mathcal{M}_\beta\left(\tilde{p}*u\dot{W}\right).
	\end{equation}
	But 
	\begin{equation}
		\mathcal{M}_\beta(\tilde{p}*Z\dot{W})=\left(\int_0^\infty
		\e^{-\beta t}\|p_t\|_{L^2(\R)}^2\,\d t\right)^{1/2}
		\cdot \mathcal{M}_\beta(Z).
	\end{equation}
	[Again one proves this first for nice $Z$'s and then take limits.] 
	Therefore,
	\begin{equation}
		\mathcal{M}_\beta\left( \tilde{p}*([\sigma\circ u]\dot{W})\right) \ge
		{\rm L}_\sigma\sqrt{\Upsilon(\beta)}
		\cdot \mathcal{M}_\beta(u).
	\end{equation}
	Combine this with \eqref{Mu} to find that
	\begin{equation}
		\left[\mathcal{M}_\beta(u) \right]^2
		\ge \left[\mathcal{M}_\beta (P_\bullet \mu)\right]^2
		+{\rm L}_\sigma^2\Upsilon(\beta)\left[\mathcal{M}_\beta(u)\right]^2.
	\end{equation}
	Now suppose, to the contrary, that $\Upsilon(\beta)\ge {\rm L}_\sigma^{-2}$.
	Then, it follows that $\mathcal{M}_\beta(P_\bullet \mu)=0$ regardless of
	the value of $m$; i.e., for all $m>0$,
	\begin{equation}
		\int_0^\infty\e^{-\beta t}\,\d t\int_{-\infty}^\infty\e^{-|x|/m}\,\d x\
		|(P_t\mu)(x)|^2 =0.
	\end{equation}
	Let $m\uparrow\infty$ and apply the monotone convergence theorem,
	and then the Plancherel theorem, in order to deduce that
	\begin{equation}\begin{split}
		0 & = \int_0^\infty\e^{-\beta t}\| P_t\mu \|_{L^2(\R)}^2\,\d t\\
		&= \frac{1}{2\pi}\int_0^\infty\e^{-\beta t}\,\d t\int_{-\infty}^\infty
			\d\xi\ \e^{-2t\Re\Psi(\xi)}|\hat\mu(\xi)|^2\\
		&=\frac{1}{2\pi}\int_{-\infty}^\infty 
			\frac{|\hat\mu(\xi)|^2}{\beta+2\Re\Psi(\xi)}\,
			\d\xi.
	\end{split}\end{equation}
	Since $\Psi$ is never infinite,
	the preceding implies that $\mu\equiv 0$, which is a contradiction. 
	It follows that $\Upsilon(\beta) < {\rm L}_\sigma^{-2}$.
\end{proof}

Theorem \ref{th:2} implies also that  Condition \ref{cond:beta:large} is sharp: 
Consider the case that 
${\rm Lip}_\sigma = {\rm L}_\sigma$. 
[This is the case, for instance, for the parabolic Anderson problem
where $\sigma(x) \propto x$, or when $\beta$ has sharp linear growth.] 
Then in this case Theorem \ref{th:1} and Theorem \ref{th:2} together imply 
that \eqref{cond:beta:large} is a necessary and sufficient condition 
for the existence of a weak solution to \eqref{heat} that has values in
$\cap_{m\ge 1} \bm{B}^k_{\beta,\eta_m}$.

\section{Examples and Remarks.}\label{sec:last}

\begin{example}[A parabolic Anderson model]
	Let $\sigma(x) = \lambda x$. In that case, the solution $u$ 
	corresponds to the conditional expected density at time $t \geq 0$ 
	of a branching L\'evy process starting with distribution $u_0$, given
	white-noise random branching. 
	The case that $\sigma(0) = 0$ and $u_0$ is a function with 
	compact support is studied in \cite{Conus_Khoshnevisan}, 
	in which intermittency properties are derived. 
	Here, $u_0$ can be a compactly supported measure (not necessarily a function). 
	If we let $\eta$ denote the one-dimensional Lebesgue measure,
	then \eqref{cond:u0} becomes
	\begin{equation} \label{anderson}
		\int_{-\infty}^{\infty} \frac{|\hat{\mu}(\xi)|^2}{\beta + 2 \Re\Psi(\xi)} \, \d\xi < \infty.
	\end{equation}
	For instance, if $\mu = \delta_0$, then condition \eqref{anderson} 
	is precisely Dalang's condition \eqref{Dalang}, and \eqref{heat} admits a weak solution. 
	In this way we can now define properly the solution of the parabolic Anderson 
	model with $u_0 = \delta_0$, which was studied in Bertini and Cancrini \cite{BC}.
	\qed
\end{example}

\begin{remark}[A nonlinear stochastic wave equation]
It is possible to apply similar techniques to the
study of the following nonlinear stochastic wave equation driven by the
Laplacian:
\begin{equation}\label{wave}
	\frac{\partial^2}{\partial t^2}u_t(x)= \kappa^2 (\Delta u_t)(x) +
	\sigma(u_t(x))\dot{W}(t\,,x) \qquad (t \geq 0, x \in \R).
\end{equation}
If the initial conditions $u_0$ and $v_0$ are nice functions, then the solution to \eqref{wave} can be written as
\begin{equation} \label{mild_wave}\begin{split}
	&u_t(x) = (\Gamma'_t * u_0)(x) + (\Gamma_t * v_0)(x)\\
	&\hskip1.8in+ \int_{[0,t]\times\R} \Gamma_{t-s}(y-x) \sigma(u_s(y))\,W(\d s\,\d y),
\end{split}\end{equation}
where $\Gamma$ is the fundamental solution of the 1-dimensional wave equation, namely
\begin{equation} \label{eq:green_wave}
	\Gamma_t(x) := \frac{1}{2} \1_{[-\kappa t,\kappa t]}(x)\qquad\text{for
	$t>0$ and $x\in\R$},
\end{equation}
and $\Gamma'_t$ denotes the weak spatial derivative of $\Gamma_t$.
Then, the existence and uniqueness of a weak solution to
\eqref{wave} in the case 
that $u_0$ and $v_0$ are (possibly signed) Borel measures 
on $\R$ can be established using the techniques of this paper,
since the definition of the generalized stochastic convolution 
applies as well with the 1-D wave propagator $\Gamma$ above.
The conditions on the initial conditions have to be adapted to 
insure that $\Gamma_t'*u_0$ and $\Gamma_t*v_0$ both are 
in $\bm{B}^k_{\beta,\eta}$, but are similar to \eqref{cond:u0}. 
The details on this are left to the reader, as the 
stochastic wave equation in dimension one has been 
widely studied already 
\cite{Cabana,CarmonaNualart,DalangCourse,Peszat,PZ,walsh,WalshNumerics}.
\end{remark}

\begin{small}

\noindent\\[2mm]
\emph{Authors' Address:}
Department of Mathematics, The University of Utah,
155 South 1400 East, JWB 233, Salt Lake City, Utah 84105--0090\\
\emph{Emails:} {\tt conus@math.utah.edu} \& {\tt davar@math.utah.edu}

\end{small}

\end{document}